\documentclass{amsart}
\usepackage{amsfonts,amscd,amssymb,amsmath,amsthm}
\usepackage[T1]{fontenc}
\usepackage{graphicx}
\usepackage{hyperref}
\usepackage{paralist}
\begin{document}

\newtheorem{theorem}{Theorem}[section]
\newtheorem{lemma}[theorem]{Lemma}
\newtheorem{corollary}[theorem]{Corollary}
\newtheorem{conjecture}[theorem]{Conjecture}
\newtheorem{proposition}[theorem]{Proposition}
\theoremstyle{definition}
\newtheorem{definition}[theorem]{Definition}
\newtheorem{example}[theorem]{Example}
\newtheorem{claim}[theorem]{Claim}
\newtheorem*{ert}{Erd\H{o}s--R\'enyi theorem}
\newtheorem*{cvt}{Cohomology vanishing theorem}
\newtheorem*{sgt}{Spectral gap theorem}
\newtheorem*{wht}{Wielandt--Hoffman theorem}
\newtheorem*{Zuk}{Spectral criterion for property (T)}

\newcommand{\EL}{{ L}}
\newcommand{\R}{\mathbb{R}}
\newcommand{\Q}{\mathbb{Q}}
\newcommand{\Z}{\mathbb{Z}}
\newcommand{\homol}{\widetilde{H}}
\newcommand{\prob}{\mathbb{P}}
\newcommand{\expect}{\mathbb{E}}
\newcommand{\var}{\mbox{Var}}
\newcommand{\cov}{\mbox{Cov}}
\newcommand{\lk}{\mbox{lk}}
\newcommand{\Pois}{\mbox{Pois}}
\newcommand{\Ns}{N_{\sigma}}
\newcommand{\prop}{\mathcal{P}}

\title[]{Sharp vanishing thresholds for cohomology of random flag complexes}
\author{Matthew Kahle}
\email{Matthew Kahle <mkahle@math.ohio-state.edu>}
\address{Department of Mathematics \\ Ohio State University\\ Columbus OH \\ 43202}
\thanks{The author acknowledges NSA grant \#H98230-10-1-0227 for partial support.}
\date{\today}
\maketitle

\begin{abstract} For every $k \ge 1$, the $k$th cohomology group $H^k(X, \Q)$ of the random flag complex $X \sim X(n,p)$ passes through two phase transitions: one where it appears, and one where it vanishes. We describe the vanishing threshold and show that it is sharp. Using the same spectral methods, we also find a sharp threshold for the fundamental group $\pi_1(X)$ to have Kazhdan's property~(T). Combining with earlier results, we obtain as a corollary that for every $k \ge 3$ there is a regime in which the random flag complex is rationally homotopy equivalent to a bouquet of $k$-dimensional spheres.

\end{abstract}

\section{Introduction}

The edge-independent random graph $G(n,p)$ is a model of fundamental importance in combinatorics, probability, and statistical mechanics. Sometimes called the Erd\H{o}s--R\'enyi model, this is defined as the probability distribution over all graphs on vertex set $[n] := \{ 1, 2, \dots, n \}$ where every edge is included with probability $p$, jointly independently --- in other words for every graph $G$ on vertex set $[n]$ with $e$ edges,
$$\prob(G) = p^e (1-p)^{ {n \choose 2} - e}.$$
We use the notation $G \sim G(n,p)$ to indicate that $G$ is a graph chosen according to this distribution.

In random graph theory, one is usually concerned with asymptotic behavior as $n \to \infty$, and it is often convenient to write $p$ as a function of $n$. For a graph property $\prop$ we say that $G \in \prop$ {\it with high probability (w.h.p.)}, if $$\prob( G \in \mathcal{P}) \to 1$$ as the number of vertices $n \to \infty$. Throughout this article, whenever we use big-$O$, little-$o$ asymptotic notation, it is also always understood as the number of vertices $n \to \infty$.

\medskip

For a monotone graph property $\mathcal{P}$, i.e. a property closed under addition of edges, a function $\bar{p}$ is said to be a {\it sharp threshold for property $\prop$} if for every fixed $\epsilon > 0$, whenever $p \ge (1+ \epsilon) \bar{p}$, w.h.p.\ $G \in \mathcal{P}$, and whenever $p \le (1 - \epsilon) \bar{p}$, w.h.p.\ $G \notin \mathcal{P}.$



In a seminal theorem in random graph theory, Erd\H{o}s and R\'enyi exhibited a sharp threshold for connectivity \cite{ER}.


\clearpage

\begin{ert}\label{thm:ER} Suppose that $\epsilon > 0$ is fixed and $G \sim G(n,p)$. 
\begin{enumerate}
\item If $$p \ge  \frac{(1 + \epsilon) \log{n}}{n},$$
 then w.h.p.\ $G$ is connected,
 \item and if  $$p \le  \frac{(1 - \epsilon) \log{n} }{n},$$ then w.h.p. $G$ is disconnected.
\end{enumerate}
\end{ert}

\bigskip


A {\it flag complex} is a simplicial complex which is maximal with respect to its underlying $1$-skeleton.  This is also sometimes called a {\it clique complex} since the faces of the simplicial complex correspond to complete subgraphs.  For a graph $H$, let $X(H)$ denote the associated flag complex.  

We are interested here in the expected topological properties of the flag complex of a random graph. Define $X(n,p)$ to be the probability distribution over flag complexes on vertex set $[n]$ where the distribution on the $1$-skeleton agrees with $G(n,p)$. We use the notation $X \sim X(n,p)$ to mean a flag complex chosen according to this distribution. This puts a measure on a wide range of possible topologies --- indeed, every simplicial complex is homeomorphic to a flag complex, e.g.\ by barycentric subdivision.

\bigskip 

The following is our main result, which may be seen as a generalization of the Erd\H{o}s--R\'enyi theorem to higher dimensions.

\begin{theorem} \label{thm:main} Let $k \ge 1$ and $\epsilon > 0$ be fixed, and $X \sim X(n,p)$. 
\begin{enumerate}
\item If $$p \ge  \left( \frac{\left(\frac{k}{2} + 1+\epsilon \right)\log{n}}{n} \right)^{1/(k+1)},$$
 then w.h.p.\ $H^k(X, \Q)  =0$,
 \item and if  $$ \left( \frac{k+1 + \epsilon}{n} \right) ^{1/k} \le p \le  \left( \frac{\left(\frac{k}{2} + 1-\epsilon \right)\log{n}}{n}\right)^{1/(k+1)},$$ then w.h.p.\ $H^k(X, \Q)  \neq 0.$
\end{enumerate}
\end{theorem}

By the universal coefficient theorem, $H^k(X, \Q)$ is isomorphic to $H_k( X, \Q)$, so these results apply for $k$th homology as well.
%

\medskip

We see immediately that each cohomology group $H^k$ passes through two phase transitions: one where nontrivial cohomology appears, and one where it disappears. The correct exponent for the first phase transition was found earlier in \cite{clique} and this article is mostly concerned with the second phase transition. It still seems reasonable to describe this threshold to be ``sharp'' in the sense described above and analogously to the Erd\H{o}s--R\'enyi theorem, even though the property of $k$th cohomology vanishing is not monotone for $k \ge 1$.

\medskip  

Using the same methods, we also exhibit a sharp threshold for the fundamental group $\pi_1(X)$ to have Kazhdan's property~(T).

\begin{theorem} \label{thm:propT}
Let $\epsilon > 0$ be fixed and $X \sim X(n,p)$.
\begin{enumerate}
\item If 
$$p \ge  \left( \frac{\left(\frac{3}{2} +\epsilon \right)\log{n}}{n} \right)^{1/2},$$
 then w.h.p.\ $\pi_1(X)$ has property~(T), and if
\item $$ \frac{1 + \epsilon}{n } \le p \le  \left( \frac{\left(\frac{3}{2}-\epsilon \right)\log{n}}{n}\right)^{1/2},$$
then w.h.p.\ $\pi_1(X)$ does not have property~(T).
\end{enumerate}
\end{theorem}

\medskip

Combining Theorem \ref{thm:main} with several earlier results discussed below, we obtain the following corollary.

\begin{corollary} {\label{cor:ddim}} 
Let $X \sim X(n,p)$.
\begin{enumerate}
\item Let $k \ge 1$ and $\epsilon > 0$ be fixed.  If
$$ \left( \frac{ \left( \frac{k}{2}+1 + \epsilon \right) \log n }{n} \right)^{1/k}  \le p \le \frac{1}{n^{1/(k+1)+ \epsilon}},$$
then w.h.p. $$\widetilde{H}_i (X, \Q)  = 0 \text{ unless } i = k,$$
in which case  $$\widetilde{H}_i (X, \Q) \neq 0.$$
\item Let $k \ge 3$ and $\epsilon > 0$ be fixed.
If $$ \left( \frac{ \left(C_k + \epsilon \right) \log n }{n} \right)^{1/k}  \le p \le \frac{1}{n^{1/(k+1)+ \epsilon}},$$
where $C_3 = 3$ and $C_k =   k/2+1$ for $k > 3$, then w.h.p.\ $X$ is rationally homotopy equivalent to a bouquet of $k$-dimensional spheres.
\end{enumerate}
\end{corollary}

The rational homotopy statement follows from Theorem \ref{thm:main} and the $k=1$ case of the earlier Theorem \ref{thm:prev1} by standard results in rational homotopy theory, and in particular by Serre's generalizations of the Hurewicz and Whitehead theorems \cite{Serre53}; see for example Wofsey's explanation on Mathoverflow \cite{waffle}. The reason for $C_3 = 3$ is that this is sufficient to ensure that $\pi_1(X)=0$ w.h.p.

For comparison, standard results on the size of maximal cliques in random graphs give that for $p$ as in Corollary \ref{cor:ddim}, w.h.p.\ the dimension $d = \dim X$ of the complex itself is either $2k$ or $2k+1$, where $d=2k$ w.h.p.\ if $p = o \left( n^{-2/(2k+1) } \right)$ and $d=2k+1$ w.h.p.\ if $p = \omega \left( n^{-2/(2k+1)} \right)$. In other words, for ``most'' choices of $p$, w.h.p.\ the only nontrivial rational homology of a random $d$-dimensional flag complex is in middle degree $\lfloor d/2 \rfloor$.


\bigskip

\subsection{Earlier work}

Recall that a topological space $T$ is said to be {\it $k$-connected} if $\pi_i (T) = 0 $ for $i \le k$.

\begin{theorem} \label{thm:prev1}  (Theorem 3.4 in \cite{clique}) Suppose that $k \ge 1$ and $\epsilon >0 $ are fixed and  $X \sim X(n,p)$.
If $$p \ge \left( \frac{(2k+1 + \epsilon)\log n}{n} \right)^{1/(2k+1)}$$ then w.h.p. $X$ is $k$-connected.
\end{theorem}

By the Hurewicz Theorem, if $X$ is $k$-connected then $\widetilde{H}_i(X, \Z) = 0$ for $i \le k$. By the universal coefficient theorem, in this case $H^i(X, \Q) = 0$ for $1 \le i \le k$. So part (1) of Theorem \ref{thm:main} improves on the vanishing threshold for $H^i(X, \Q)$, in particular substantially improving the exponent from $1/(2k+1)$ to $1/(k+1)$, which is best possible for rational cohomology.

On the other hand, the exponent $1/3$ gives the rough vanishing threshold for $\pi_1(X)$, as shown by Babson.

\begin{theorem} \label{thm:bab} (Theorem 1.1 in \cite{babson})
If $X \sim X(n,p)$ where $\epsilon > 0 $ is fixed and $$\frac{1 + \epsilon}{n} \le  p \le \frac{1}{n^{1/3+ \epsilon}}$$ then w.h.p.\ $\pi_1(X)$ is a nontrivial hyperbolic group.
\end{theorem}

Theorem \ref{thm:bab} is closely related to the results in \cite{bhk}, where a parallel result is shown for Bernoulli random $2$-complexes studied earlier by Linial and Meshulam \cite{LM}.

\medskip

The following earlier result can be compared with part (2) of Theorem \ref{thm:main}.

\begin{theorem} \label{thm:prev2}  (Theorem 3.8 in \cite{clique}) Suppose that $k \ge 1$ and $X \sim X(n,p)$.
If $$\omega \left( \frac{1}{n^{1/k}}\right) \le p \le o \left( \frac{1}{n^{1/(k+1)}} \right)$$ then w.h.p. $H^k (X, \Q) \neq 0$. 
\end{theorem}

The upper bound in part (2) of Theorem \ref{thm:main} improves the earlier upper bound from Theorem \ref{thm:prev2} by a roughly logarithmic factor to be essentially best possible. The lower bound in part (2) of Theorem \ref{thm:main} is also a slight improvement on the earlier lower bound in Theorem \ref{thm:prev2}, and the following earlier result shows that the exponent $1/k$ in these lower bounds can not be improved.

\begin{theorem} \label{thm:prev3} (Theorem 3.6 in \cite{clique}) Suppose that $k \ge 1$ and $\epsilon > 0$ are fixed.
If $$p \le  \frac{1}{n^{1/k  + \epsilon}},$$
then w.h.p.\ $H^k (X, \Q) = 0$.
\end{theorem}
%

\medskip

The proof of Theorem \ref{thm:main} is based on earlier work in group cohomology by Garland \cite{Garland}, and refinements due to Ballman and \'Swi\k{a}tkowski \cite{Ballmann}. See also the S{\'e}minaire Bourbaki by Borel \cite{Borel}, and work of \.Zuk \cite{z03} on thresholds for property~(T) in random groups.

The proof also depends in an essential way on recent work on spectral gaps of Erd\H{o}s--R\'enyi random graphs by Hoffman et al.\ \cite{hkp}, in particular the spectral gap theorem in Section \ref{sec:pfmain}.

\bigskip

The outline for the rest of the paper is as follows. In Section \ref{sec:cliques} we make preliminary calculations on the number of maximal $k$-cliques in random graphs. In Sections \ref{sec:pfmain} and Section \ref{sec:pfmain2} we prove Theorems \ref{thm:main} and \ref{thm:propT}.  In Section \ref{sec:comments} we close with comments and conjectures.

\section{Preliminary calculations for maximal $(k+1)$-cliques} \label{sec:cliques}

We use the standard notation $[n] = \{ 1, 2, \dots, n \}$, and then $${[n] \choose m} =\{ \{1, 2, \dots, m \}, \dots  \}$$ denotes the {\it set} of $m$-subsets of $[n]$, a set of cardinality $n \choose m$.

Let $N_{k+1}$ denote the number of {\it maximal} $(k+1)$-cliques, i.e. $(k+1)$-cliques which are not contained in any $(k+2)$-cliques.  We write $N_{k+1}$ as a sum of indicator random variables, as follows.  For $i \in {[n] \choose k+1}$  let $A_i$ be the event that the vertex set corresponding to $i$ spans a maximal $(k+1)$-clique, and let $Y_i$ be the indicator random variable for the event $A_i$.  Then
$$ N_{k+1} = \sum_{ i \in {[n] \choose k+1} }Y_i.$$

Since the probability that $i$ spans a $(k+1)$-clique is $p ^{k+1 \choose 2}$, and the probability of the independent event that the vertices in $i$ have no common neighbor is $(1-p^{k+1})^{n-k-1}$, we have $$E[Y_i] =  p ^{k+1 \choose 2} (1-p^{k+1})^{n-k-1}.$$
By linearity of expectation we have
$$E[N_{k+1}] = {n \choose k+1} p ^{k+1 \choose 2} (1-p^{k+1})^{n-k-1}.$$

Now suppose  $$p = \left( \frac{ \left( \frac{k}{2}+1 \right) \log{n} + \left( \frac{k}{2} \right) \log\log{n} + c}{n} \right)^{1 / (k+1)},$$
where $c \in \R$ is constant. Then in this case,

\begin{align*}
E[N_{k+1}] & = \sum_{i \in {[n] \choose k+1} }E[Y_i]\\
& = { n \choose k+1}  p ^{k+1 \choose 2} (1-p^{k+1})^{n-k-1}\\
& \approx \frac{n^{ k+1}}{(k+1)!}  p ^{k+1 \choose 2} e^{-p^{k+1}n}\\
& =  \frac{n^{ k+1}}{(k+1)!} \left( \frac{ (\frac{k}{2}  + 1 + o(1) ) \log n}{n} \right)^{k/2} n^{-(k/2+1)} (\log n)^{-k/2} e^{-c},\\
\end{align*}
and then 
\begin{equation} \label{eq:ex}
E[N_{k+1}] \to \frac{ (\frac{k}{2}+1)^{k/2}}{(k+1)!} e^{-c},
\end{equation}
as $n \to \infty$.


\subsection{Zero expectation}

Letting $c \to \infty$ in Equation (\ref{eq:ex}) gives that $E[N_{k+1}] \to 0$.  By Markov's inequality, we conclude the following.

\begin{lemma} \label{lem:bigp} Let $G \sim G(n,p)$, and $N_{k+1}$ count the number of maximal $(k+1)$-cliques in $G$.
If $$ p \ge \left( \frac{ \left( \frac{k}{2}+1 \right) \log{n} + \left( \frac{k}{2} \right) \log\log{n} + \omega(1)}{n} \right)^{1 / (k+1)},$$ 
then $N_{k+1}=0$ w.h.p.
\end{lemma}
 
\subsection{Infinite expectation}

Suppose that $$ \omega \left( \frac{1}{n^{2/k}} \right) \le p \le \left( \frac{ \left( \frac{k}{2}+1 \right) \log{n} + \left( \frac{k}{2} \right) \log\log{n} - \omega(1)}{n} \right)^{1 / (k+1)}.$$ 
In this case we have that $E[N_{k+1}] \to \infty$.  By Chebyshev's inequality, if we also have $\var[N_{k+1}] = o \left( E[N_{k+1}]^2 \right),$ then $$\prob[ N_{k+1} > 0] \to 1.$$ (See for example, Chapter 4 of \cite{Alon}.)

So once we bound the variance we have the following.

\begin{lemma} \label{lem:smallp} Let $0 < \epsilon < \frac{1}{k(k+1)}$ be fixed, and $G \sim G(n,p)$.
If $$\frac{1}{n^{1/k - \epsilon}} \le p \le \left( \frac{ \left( \frac{k}{2} + 1 \right) \log{n} + \left( \frac{k}{2} \right) \log\log{n} - \omega(1)}{n} \right)^{1 / (k+1)},$$ then $N_{k+1} > 0$ w.h.p.
\end{lemma}

\begin{proof}[Proof of Lemma \ref{lem:smallp}]
As above, write $N_{k+1}$ as a sum of indicator random variables.

$$ N_{k+1} = \sum_{ i \in {[n] \choose k+1} }Y_i.$$

Then $$\var [ N_{k+1} ] \le E[N_{k+1}] + \sum_{i, j \in {[n] \choose k+1} } \cov[Y_i, Y_j] $$ where the covariance is
\begin{align*}
 \cov[Y_i, Y_j] & = E[Y_i Y_j ] - E[Y_i]E[Y_j]\\
 & = \prob[A_i \mbox{ and } A_j] - \prob[A_i] \prob[A_j],
\end{align*}
since $Y_i$ are indicator random variables.

Let $I = I_{i,j} = | i \cap j|$ be the number of vertices in the intersection of subsets $i$ and $j$.  It is convenient to divide into cases depending on the cardinality of $0 \le I < k+1$.

\bigskip

\noindent {\it Case I: 
$$I = 0$$} 

Given two disjoint subsets, $i,j \in {[n] \choose k+1}$,
\begin{align*}
\prob[A_i \mbox{ and } A_j] & = p^{2{k+1 \choose 2}}(1-2p^{k+1}+p^{2k+2} )^{n-2k-2}\left(1 - O \left( p^k \right) \right),\\
\end{align*}
and
\begin{align*}
\prob [A_i] \prob[A_j] &=  \left( p^{{k+1 \choose 2}} ( 1- p^{k+1})^{n-k-1} \right)^2\\
& = p^{2{k+1 \choose 2}} \left( 1- 2p^{k+1}+p^{2k+2}\right)^{n-k-1},\\
& = p^{2{k+1 \choose 2}} \left( 1- 2p^{k+1}+p^{2k+2}\right)^{n-2k-2} \left(1 - 2p^{k+1}+p^{2k+2} \right)^{k+1},\\
& = p^{2{k+1 \choose 2}} \left( 1- 2p^{k+1}+p^{2k+2}\right)^{n-2k-2} \left(1 - O \left( p^{(k+1)^2} \right) \right),\\
\end{align*}
so
\begin{align*}
\prob[A_i \mbox{ and } A_j] - \prob[A_i] \prob[A_j] & = p^{2{k+1 \choose 2}} ( 1- 2p^{k+1}+p^{2k+2})^{n-2k-2} O \left( p^k \right).\\
\end{align*}

The number of vertex-disjoint pairs $i,j$ is $O \left( n^{2k+2} \right)$ so the total contribution $S_0$ to the variance is
$$S_0 = O \left( n^{2k+2}   p^{2{k+1 \choose 2}} ( 1- 2p^{k+1}+p^{2k+2})^{n-k-1} p^k  \right)$$
Compare this to $$E [N_{k+1}]^2 = {n \choose k+1}^2  p^{2{k+1 \choose 2}} ( 1- p^{k+1})^{2(n-k-1)}.$$
Clearly $$ S_0 /  E [N_{k+1}]^2  = O \left (p^k \right),$$
and since $p \to 0$ by assumption, we have that $$S_0 = o \left( E [N_{k+1}]^2 \right),$$ as desired.\\

\noindent{\it Case II: $$I=1$$}

If $I=1$ then
\begin{align*}
\prob[A_i \mbox{ and } A_j] & = p^{2{k+1 \choose 2}}(1-2p^{k+1}+p^{2k+1} )^{n-2k-1}(1 - O(p^k)),\\
\end{align*}
and
\begin{align*}
\prob [A_i] \prob[A_j] &=  \left( p^{{k+1 \choose 2}} ( 1- p^{k+1})^{n-k-1} \right)^2\\
& = p^{2{k+1 \choose 2}} \left( 1- 2p^{k+1}+p^{2k+2}\right)^{n-k-1},\\
& = p^{2{k+1 \choose 2}} \left( 1- 2p^{k+1}+p^{2k+2}\right)^{n-2k-1} \left( 1- 2p^{k+1}+p^{2k+2}\right)^{k}\\
& = p^{2{k+1 \choose 2}} \left( 1- 2p^{k+1}+p^{2k+2}\right)^{n-2k-1} \left( 1- O \left( p^{k(k+1)} \right) \right).\\
\end{align*}
Subtracting, we have 
\begin{align*}
\prob[A_i \mbox{ and } A_j] - \prob[A_i] \prob[A_j] & = p^{2{k+1 \choose 2}} ( 1- 2p^{k+1}+p^{2k+2})^{n-2k-1} O\left(p^k \right).\\
\end{align*}
There are $O \left(n^{2k+1}\right)$ such pairs of events, so
$$S_1 = O \left( n^{2k+1}p^{2{k+1 \choose 2}} ( 1- 2p^{k+1}+p^{2k+2})^{n-2k-1} p^k \right).$$
Compare this to $$E [N_{k+1}]^2 = {n \choose k+1}^2  p^{2{k+1 \choose 2}} ( 1- p^{k+1})^{2(n-k-1)}.$$
Now $$ S_1 /  E [N_{k+1}]^2  = O \left( n^{-1} p^k \right) =o(1),$$
since $n \to \infty$ and $p \to 0$. So we have that $$S_1 = o \left( E [N_{k+1}]^2 \right),$$ as desired.\\

\noindent {\it Case III: 
$$2 \le I \le k$$}

In this case, 
\begin{align*}
\prob[A_i \mbox{ and } A_j] & = p^{2{k+1 \choose 2 }-{I \choose 2} }(1-2p^{k+1}+p^{2k+2-I} )^{n-2k-2+I}(1 - O(p^k)),\\
\end{align*}
and
\begin{align*}
\prob [A_i] \prob[A_j] &=  \left( p^{{k+1 \choose 2}} ( 1- p^{k+1})^{n-k-1} \right)^2\\
& = p^{2{k+1 \choose 2}} ( 1- 2p^{k+1}+p^{2k+2})^{n-k-1}.\\
\end{align*}

Comparing, we have 
\begin{align*}
\frac{\prob [A_i] \prob[A_j] }{\prob[A_i \mbox{ and } A_j] } & \le p^{I \choose 2} \left( 1 + \frac{p^{2k+2} - p^{2k+2-I}}{1-2p^{k+1} + p^{2k+2-I}} \right)^n  \left(1 + o(1) \right) \\
& \le p^{I \choose 2},\\
\end{align*}
and since $p \to 0$ and $I \ge 2 $ by assumption, $$\frac{\prob [A_i] \prob[A_j] }{\prob[A_i \mbox{ and } A_j] } \to 0.$$

So $$\prob[A_i \mbox{ and } A_j]  - \prob [A_i] \prob[A_j] = \left( 1-o(1) \right) \prob[A_i \mbox{ and } A_j],$$
and now we bound the covariance $$ \cov[Y_i, Y_j]$$ by bounding the probability $\prob[A_i \mbox{ and } A_j]$.

For every $2 \le I < k+1$, there are $O \left( n^{2k+2-I} \right)$ pairs of events $i, j$ with vertex intersection of cardinality $I$.

So the total contribution to variance from such pairs is at most
$$S_I = O \left( n^{2k+2-I}   p^{2{k+1 \choose 2} -{I \choose 2}} (1-2p^{k+1}+p^{2k+2-I} )^{n-2k-2+I} \right).$$

Compare this to $$E [N_{k+1}]^2 = {n \choose k+1}^2  p^{2{k+1 \choose 2}} ( 1- p^{k+1})^{2(n-k-1)}.$$
We have
\begin{align*}
S_I / E [N_{k+1}]^2 & = O \left( n^{-I} p^{-{I \choose 2}} \right).\\
\end{align*}
Clearly 
\begin{align*}
 n^I p^{I \choose 2} & = \left( n p^{(I-1) / 2} \right)^I  \to \infty\\
\end{align*}
as $n \to \infty$, since $I \le k$ and $p = \omega( n^{-1 / (k+1) } )$.
Hence $$S_I = o \left( E [N_{k+1}]^2 \right)$$
for $2 \le I \le k$.

\end{proof}

\subsection{Finite expectation} 

By computing the factorial moments of $N_{k+1}$, the following limit theorem can be proved. (See for example Section 6.1 of \cite{randomgraphs}.)
\begin{theorem} \label{thm:midp}
If
$$ p = \left( \frac{ \left( \frac{k}{2} + 1 \right) \log{n} + \left( \frac{k}{2} \right) \log\log{n} + c}{n} \right)^{1 / (k+1)},$$ where $c \in \R$ is constant, then the number  $N_{k+1}$ of maximal $(k+1)$-cliques approaches a Poisson distribution $$N_{k+1} \to \Pois ( \mu)$$ with mean
$$\mu =  \frac{ (k/2+1)^{k/2}}{(k+1)!} e^{-c}.$$  
\end{theorem}

Since we do not use Theorem \ref{thm:midp} for anything else, we state it without proof.  We record the combinatorial observation for the sake of completeness, however, and also to provide some justification for a conjecture in Section \ref{sec:comments}.

\section{Vanishing cohomology and property~(T)} \label{sec:pfmain}

In this section we prove a slightly sharper version of part (1) of Theorem \ref{thm:main} and Theorem \ref{thm:propT}.
Set 
$$\bar{p} = \left( \frac{ \left( \frac{k}{2} + 1 \right) \log{n} + C_k \sqrt{\log n}\log\log{n} }{n} \right)^{1 / (k+1)},$$
where $C_k$ is a constant depending only on $k$, to be chosen later,
and we assume that $p \ge \bar{p}$.

\bigskip

For a finite graph $H$, let $C^0(H)$ denote the vector space of $0$-forms on $H$, i.e.\ the vector space of functions $f \colon V(H) \to \R$.
If all the vertex degrees are positive then the averaging operator $A$ on $C^0(H)$ is defined by
$$Af(x) = \frac{1}{\deg x} \sum_{y \sim x} f(y),$$
where the notation $y \sim x$ means that the sum is over all vertices $y$ which are adjacent to vertex $x$.
The identity operator on $C^0(H)$ is denoted by $I$.
Then the {\it normalized graph Laplacian} $\mathcal{L} = \mathcal{L} (H)$ is a linear operator on $C^0(H)$ defined by 
$\mathcal{L} = I - A$.

The eigenvalues of $\mathcal{L}$ satisfy $0 = \lambda_1 \le \lambda_2 \le \dots \le \lambda_N \le 2$, where $N = |V(G)|$ is the number of vertices of $H$.  Moreover, the multiplicity of the zero eigenvalue equals the number of connected components of $H$.  In the case that $H$ is connected then the smallest positive eigenvalue $ \lambda_2 [ H ]$, is sometimes called the {\it spectral gap} of $H$.

A simplicial complex $\Delta$ is said to be {\it pure $D$-dimensional} if every face of $\Delta$ is contained in a $D$-dimensional face.  A special case of Theorem 2.1 in \cite{Ballmann} is the following.

\begin{cvt} [Garland, Ballman--\'Swi\k{a}tkowski] \label{thm:BS}  Let $\Delta$ be a pure $D$-dimensional finite simplicial complex such that for every $(D-2)$-dimensional face $\sigma$, the link $\lk_{\Delta}(\sigma)$ is connected and has spectral gap $$\lambda_2[ \lk_{\Delta}(\sigma)] > 1 - \frac{1}{D}.$$  Then $H^{D-1}(\Delta,\Q) = 0$.
\end{cvt}

The cohomology group  $H^{D-1}(\Delta, \Q)$ only depends on the $D$-skeleton of $\Delta$.  
So to use Theorem \ref{thm:BS} to show that $H^k(X, \Q) = 0$ we will show that if the edge probability $p$ is large enough, with high probability
\begin{enumerate}
\item the $(k+1)$-skeleton of $X \sim X(n,p)$ is pure-dimensional, and
\item for every $(k-1)$-dimensional face $\sigma \in X$, the link $\lk_{X}( \sigma)$ is connected and has spectral gap $$\lambda_2[\lk_{X}( \sigma) ]  > 1- \frac{1}{k}.$$
\end{enumerate}

%
%

\subsection{Pure-dimensional}

We first establish that for large enough $p$ the $(k+1)$-skeleton of $X$ is pure-dimensional. 

\begin{lemma} \label{lem:pure}
If $p \ge \bar{p}$ then w.h.p.\ the $(k+1)$-skeleton of $X \sim X(n,p)$ is pure $(k+1)$-dimensional; in other words, every face is contained in the boundary of a $(k+1)$-face. 
\end{lemma}

\begin{proof}
A $k$-face not contained in a $(k+1)$-face would correspond to a maximal $(k+1)$-clique.  But 
$$\bar{p} \ge \left( \frac{ \left( \frac{k}{2} + 1 \right) \log{n} + \left( \frac{k}{2} \right) \log\log{n} + \omega(1)}{n} \right)^{1 / (k+1)},$$
so by Lemma \ref{lem:bigp}, 
for $p \ge \bar{p}$ the probability that there exist any maximal $(k+1)$-cliques tends to zero as $n \to \infty$. The argument that for $0 \le i < k$ , w.h.p.\ every $i$-dimensional face is contained in an $(i+1)$-dimensional face is identical.
\end{proof}

\subsection{Connectedness and spectral gap} \label{sec:gap}

Finally, we must show that if $p$ is large enough, then w.h.p.\ the link of every $(k-1)$-dimensional face in the $(k+1)$-skeleton of $X$ is connected and has sufficiently large spectral gap.  We require the following spectral gap theorem, which is Theorem 2.1 in \cite{hkp}.

\begin{sgt}
\label{thm:ergap} [Hoffman--Kahle--Paquette]
Let $G \sim G(n,p)$ be an Erd\H{o}s--R\'enyi random graph. Let $\mathcal{L}$ denote the normalized Laplacian of $G$, and let $\lambda_1 \leq \lambda_2 \leq \cdots \leq \lambda_n$ be the eigenvalues of $\mathcal{L}.$  For every fixed $\alpha \geq 0$, there is a constant $C_{\alpha}$ depending only on $\alpha$, such that if
$$p \ge  \frac{(\alpha+1)\log n + C_{\alpha} \sqrt{\log n} \log\log n}{n},$$
then $G$ is connected and 
$$\lambda_2 (G) > 1- o(1),$$
with probability $1 - o(n^{-\alpha})$.
\end{sgt}

\medskip
Let $\Ns$ denote the number of vertices in the link of a $(k-1)$-dimensional face $\sigma$ in $X \sim X(n,p)$.
Most of the work in this section is in establishing the following estimate, which will allow us to apply the spectral gap theorem.

\begin{lemma} \label{lem:pbig}
If $$ p \ge \left( \frac{(k/2+1) \log n + C_k \sqrt{ \log n} \log \log n}{n} \right)^{1/(k+1)}$$
then w.h.p.\ 
$$ \frac{(\alpha +1) \log \Ns +C_{\alpha} \sqrt{ \log \Ns} \log \log \Ns }{\Ns} \le 1 / p$$
for every $(k-1)$-dimensional face $\sigma \in X$, where $\alpha = k(k+3)/2$ and $C_{\alpha}$ is as defined in the spectral gap theorem, and $C_k $ is a constant which only depends on $k$.
\end{lemma}
%


\begin{proof}[Proof of Lemma \ref{lem:pbig}]

Let $f_{k-1}$ denote the number of $(k-1)$-dimensional faces. Then $f_{k-1}$ has the same distribution as the binomial random variable $\mbox{Bin}({n \choose k},p^{k \choose 2})$ and
$\Ns$ has the same distribution as the binomial random variable $\mbox{Bin}(n-k,p^k)$. So for
$p \ge n^{-1/(k+1)}$, Chernoff bounds give that with high probability,
$$ \mu - \mu^{3/5} \le  \Ns \le \mu + \mu^{3/5}$$ for every $(k-1)$-dimensional face $\sigma$, where $$\mu  = n p^k.$$
Let $N =  \mu - \mu^{3/5}$.

Now let 
$$g(x) = \frac{ (\alpha +1 ) \log x + C_{\alpha} \sqrt{ \log x} \log \log x}{x}.$$
Then since sums and products of increasing positive functions are increasing, a little calculus shows that $g(x)$ is decreasing on the interval $x \in (16, \infty)$.

So it suffices to show that w.h.p.\
\begin{equation}
 p \ge \frac{ (\alpha +1 ) \log N + C_{\alpha} \sqrt{ \log N} \log \log N}{N},
\label{eqn:hypo1}
\end{equation}
since w.h.p.\ $\Ns  \ge N$ for every $\sigma$.

Write $$f(p) = Np - (\alpha+1) \log N - C_{\alpha} \sqrt{ \log N} \log \log N$$ 
and $$\bar{p} = \left( \frac{(k/2+1) \log n + C_k \sqrt{ \log n} \log \log n}{n} \right)^{1/(k+1)}.$$
The goal is to show that $f(p) > 0 $ for $p \ge \bar{p}$.\\

\noindent {\it Case I:
$$ (k/2+1) \log n + C_k \sqrt{ \log n} \log \log n \le np^{k+1} \le n^{1/10k}$$ }

A reasonable approximation of $f$ is given by the auxillary function
$$\tilde{f}(p) = \mu p - (\alpha+1) \log \mu - C_{\alpha} \sqrt{\log n} \log \log n - 1.$$

In particular we will show that for large $n$, and in the given range of $p$, \begin{inparaenum} \item $f(p) \ge \tilde{f}(p)$, \item $\tilde{f}(\bar{p}) > 0$, and \item $d \tilde{f}/dp > 0$, \end{inparaenum} which together establish the claim that $f(p) > 0$ for $p$ in the given range.\\

\begin{enumerate}
\item Clearly
\begin{eqnarray*}
-C_{\alpha} \sqrt{\log N} \log \log {N} \ge  -C_{\alpha} \sqrt{\log n} \log \log {n},
\end{eqnarray*}
since $n \ge N$.

Since $N \le \mu$ and $\alpha > 0$, we also have $$
-(\alpha + 1) \log N  \ge - (\alpha+ 1) \log \mu.$$

Finally, we are assuming that $np^{k+1} \le n^{1/10k}$ and $1/10k < 2 / (3k+5)$ since $k \ge 1$, so we have
\begin{align*}
np^{k+1} & \le n^{2/ (3k+5)}\\
p^{k+1} & \le n ^{- (3k+3)/(3k+5)}\\
p^{(3k+5)/5 }& \le n ^{- 3/5}\\
p(np^k)^{3/5} & \le 1,
\end{align*}
or in other words
$$Np \ge \mu p  - 1.$$

Adding the three inequalities yields $f(p) \ge \tilde{f}(p)$.\\

\item Since
$$\bar{p}  = \left( \frac{(k/2+1) \log n + C_k \sqrt{ \log n} \log \log n}{n} \right)^{1/(k+1)},$$
we have
$$\log{\bar{p}} = \frac{1}{k+1} \left( \log \log n - \log n  \right) + O(1).$$
Recalling that $$\alpha + 1 = \frac{(k+1)(k+2)}{2},$$
we have that
\begin{align*}
\tilde{f}(\bar{p}) 
& = n \bar{p}^{k+1} - (\alpha+1)( \log n  + k \log \bar{p}) - C_{\alpha} \sqrt{\log n} \log \log n - 1\\
& =  (k/2+1) \log n + C_k  \sqrt{ \log n} \log \log n - (\alpha+1)(\log n + k \log{ \bar{p}}) \\
&  \indent - C_{\alpha}  \sqrt{ \log n} \log \log n - 1\\
& =  (k/2+1) \log n + (C_k - C_{\alpha} ) \sqrt{ \log n} \log \log n \\
& \indent - \frac{(k+1)(k+2)}{2} \left(\frac{1}{k+1}\log n + \frac{k}{k+1}\log \log n \right) - O(1)\\
& = (C_k - C_{\alpha} ) \sqrt{ \log n} \log \log n  - \left( \frac{k+2}{k} \right) \log \log n - O(1),
\end{align*}
so as long as $C_k > C_{\alpha}$, we have that $\tilde{f}(\bar{p}) > 0 $ for large enough $n$.\\

\item Since
$$\tilde{f}(p) = np^{k+1} - (\alpha+1)k \log p  - C_{\alpha} \sqrt{\log n} \log \log n - 1,$$
we have
$$d \tilde{f}/dp = (k+1)n p^k - (\alpha+1)k p^{-1}.$$
Then $ d \tilde{f}/dp =0 $ only at $$p_c = \left( \frac{k (\alpha + 1)}{k+1} n \right)^{1/(k+1)}.$$
and $\alpha$ and $k$ are constant so $p_c < \bar{p}$ for large enough $n$.

Since $d \tilde{f}/dp$ is continuous on $(0, \infty)$ and $$\lim_{p \to \infty} \tilde{f}(p) = \infty,$$ we have $d \tilde{f}/dp > 0$ for $p \ge \bar{p}$.\\
\end{enumerate}

\bigskip

\noindent {\it Case II:
$$  n^{1/100k}\le np^{k+1}$$ }

In this case proving that $f(p) > 0$ is more straightforward.  Indeed,
\begin{align*}
f(p) & = Np - (\alpha +1 ) \log N - C_{\alpha} \sqrt{\log N } \log \log N\\
 & \ge ( 1- o(1))  \mu p  - (\alpha +1 ) \log n - C_{\alpha}  \sqrt{\log n } \log \log n\\
 & = ( 1- o(1)) n p^{k+1}  - (\alpha +1 ) \log n - C_{\alpha}  \sqrt{\log n } \log \log n\\
 & = n^{1/100k} - O( \log n),
 \end{align*}
 so for large enough $n$, we have $f(p) > 0$.
 \bigskip
 
Together, Cases I and II establish that $f(p) > 0$ for $p > \bar{p}$.

\end{proof}

Now we are in position to prove one implication of the main result.

\begin{proof}[Proof of part (1) of Theorem \ref{thm:main}] 

Suppose $p \ge \bar{p}$ and $X \sim X(n,p)$, and let $f_{k-1}$ denote the number of $(k-1)$-dimensional faces of $X$. Then Chernoff bounds show that w.h.p.\ 
$$f_{k-1} \le (1 + o(1) ) { n \choose k } p^{k \choose 2}.$$

Lemma \ref{lem:pbig} gives that w.h.p.
$$ p \ge  \frac{(\alpha +1) \log \Ns +C_{\alpha} \sqrt{ \log \Ns} \log \log \Ns }{\Ns}$$
for every $(k-1)$-face $\sigma$, where $\alpha = k(k+3)/2$.

The link of a $(k-1)$-face $\sigma$ in the $(k+1)$-skeleton has the same distribution as an Erd\H{o}s--R\'enyi random graph $G(N_{\sigma}, p)$, so the spectral gap theorem gives that the probability $P_{\sigma}$ that $\lambda_2 [ G(N_{\sigma}, p) ] < 1 - 1/(k+1)$ is $o(N_{\sigma}^{-\alpha})$.

Let $P_f$ be the probability that there exists a face $\sigma$ such that $$\lambda_2 [ G(N_{\sigma}, p) ] < 1-  \frac{1}{k+1}.$$
Applying a union bound, 
\begin{align*}
P_f & \le \sum_{\sigma} P_{\sigma}\\ 
& \le \sum_{\sigma} o \left( N_{\sigma}^{-\alpha} \right)\\
& \le \sum_{\sigma} o \left( \mu^{-\alpha} \right)\\
& \le (1 + o(1) ) { n \choose k } p^{k \choose 2} o(\mu^{-\alpha})\\
& =o \left(   n^k p^{k \choose 2} (np^k)^{-k(k+3)/2} \right)\\
& = o \left( \left( np^{k+1} \right)^{-k(k+1)/2} \right)\\
& = o(1),
\end{align*}
since $  np^{k+1} \to \infty$ for $p \ge \bar{p}$.

Now we have w.h.p.\ the spectral gap of the link of every $(k-1)$-face $\sigma$ in the $(k+1)$-skeleton is greater than $1/(k+1)$, so the cohomology vanishing theorem gives that $H^k(X, \Q)=0$ as desired.
\end{proof}

\begin{proof}[Proof of  part (1) of Theorem \ref{thm:propT}] 
The proof is the same as the case $k=1$ of part (1) of Theorem \ref{thm:main}, but instead of the cohomology vanishing theorem we use the following closely related theorem of \.Zuk \cite{z03}.

\begin{Zuk} \label{thm:tool}
If $X$ is a pure $2$-dimensional locally-finite simplicial complex such that for every vertex $v$, the vertex link $\lk(v)$ is connected and the normalized Laplacian $L= L [\lk(v) ]$ satisfies $\lambda_2(L) > 1/2$, then $\pi_1(X)$ has property~(T).
\end{Zuk}

Both the cohomology vanishing theorem for $k=1$ and the spectral criterion for property~(T) require that the link of every vertex in the $2$-skeleton of $X$ has spectral gap at least $1/2$, and this is exactly what was checked in the proof of part (1) of Theorem \ref{thm:main} above.
\end{proof}

\section{Non-vanishing cohomology} \label{sec:pfmain2}

In this section we prove Part (2) of Theorems \ref{thm:main} and \ref{thm:propT}.  In particular we show that if $X \sim X(n,p)$ where
$$ \left( \frac{k+1 + \epsilon}{n} \right) ^{1/k} \le p \le  \left( \frac{\left(\frac{k}{2} + 1-\epsilon \right)\log{n}}{n}\right)^{1/(k+1)},$$ then w.h.p.\ $H^k(X, \Q)  \neq 0.$
The strategy is to show that in this regime there exist $k$-faces not contained in the boundary of any $(k+1)$-face and which generate nontrivial cohomology classes. This is the higher-dimensional analogue of isolated vertices being the main obstruction to connectivity of the random graph $G(n,p)$; see for example Chapter 7 of \cite{Bollo}. 

First we show that if $p$ is in the given regime, then w.h.p.\ there exist $k$-dimensional faces $\sigma \in X$ which are not contained in the boundaries of any $(k+1)$-dimensional faces --- such faces generate cocycles in $H^k$ (i.e.\ the characteristic function of $\sigma$ is a cocycle).   Then we show that if $p$ is sufficiently large, no $k$-dimensional face can be a coboundary.  Putting it all together, we find an interval of $p$ for which there is at least one $k$-dimensional face that represents a nontrivial class in $H^k(X, \Q)$.
%
%
\subsection{Nontrivial cocycles}

Lemma \ref{lem:smallp} gives that for $p$ in this regime, w.h.p.\ there exist maximal $(k+1)$-cliques in $G \sim G(n,p)$.  But these are equivalent to isolated $k$-faces $\sigma$ in $X \sim X(n,p)$, and the characteristic functions of such $\sigma$ are cocycles.
The main point is to show that $\sigma$ nontrivial, i.e. that $\sigma$ is not the coboundary of anything.

We have showed above that there exist $k$-dimensional faces which are not contained in the boundary of any $(k+1)$-dimensional face.  Any such face generates a class in the vector space $Z^k(X)$ of $k$-cocycles.  Now we will show that in the same regime of $p$, w.h.p.\ no $k$-dimensional face represents a $k$-coboundary.  Hence $H^k(X, \Q) \neq 0$.

Consider the exact sequence of the pair $(X, X - \sigma)$ where $\sigma$ is a maximal $k$-face:
$$H^{k-1}(X-\sigma) \rightarrow H^k(X, X-\sigma) \rightarrow H^k(X)$$
By excision, $H^k(X, X- \sigma) \cong H^k( \sigma, \partial  \sigma) \cong \Q$.

Suppose that a $k$-dimensional face $\sigma \in X$ represents a $k$-coboundary, i.e.\ $\sigma = d \phi$ for some $(k-1)$-cochain $\phi$.  Then  $\phi$ represents a nontrivial class in $H^{k-1}(X- \sigma)$. (The notation $X - \sigma$ means $X$ with the open face $\sigma$ deleted.)  The following lemma shows that it is unlikely that such a $\sigma$ exists.

\begin{lemma} \label{lem:del} Fix $k \ge 1$ and $0 < \epsilon \le 1/k$, and let $X \sim X(n,p)$.  If $$p \ge \frac{1}{n^{1 / k - \epsilon}},$$ then w.h.p.\ $H^{k-1} ( X - \sigma, \Q) = 0$ for every maximal $k$-face $\sigma$.
\end{lemma}

\begin{proof}[Proof of Lemma \ref{lem:del}]
The claim that $H^{k-1} ( X, \Q) = 0$ is implied by Part (1) of Theorem \ref{thm:main} (with the index shifted by $1$), proved in Section \ref{sec:pfmain}, so our focus is on the second part of the claim, that $H^{k-1} ( X-\sigma, \Q) = 0$ for every $k$-face $\sigma$.

We apply the spectral gap theorem again.  Since the proof here is so similar to that of Section \ref{sec:pfmain}, we omit some details and focus on what is new in this argument.

We may restrict our attention to the $k$-skeleton of $X$.   Let $\sigma$ be an arbitrary $k$-dimensional face of $X$. Let $\tau$ be a $(k-2)$-dimensional face of $X - \sigma$, and denote the link of $\tau$ in $X-\sigma$ by $\lk_{ X - \sigma}(\tau)$.  Since we are restricting our attention to the $k$-skeleton of $X$, this link is a graph.  Clearly, either $\lk_{ X - \sigma}(\tau)=\lk_{X}(\tau)$ or $\lk_{ X - \sigma}(\tau)= \lk_{X}(\tau) - e$ for some edge $e$ in the graph $\lk_{X}(\tau)$.  

We have control on the spectral gap of $\lk_{X}(\tau)$ by the spectral gap theorem.  From this we can control the spectral gap of $\lk_{X-\sigma}(\tau)$ via the Wielandt--Hoffman theorem \cite{HW53}.

\begin{wht}  Let $A$ and $B$ be normal matrices.  Let their eigenvalues $a_i$ and $b_i$ be ordered such that $\sum_i |a_i - b_i|^2$ is minimized.  Then we have $$ \sum_i |a_i - b_i|^2 \le \|A-B\| ,$$
where $\| \cdot   \|$ denotes the Frobenius matrix norm.
\end{wht}

Consider the normalized Laplacians $A=\mathcal{L}[ \lk_{X}(\tau)]$ and $B= \mathcal{L} [\lk_{X - \sigma}(\tau)]$ --- since these matrices are symmetric, they are normal.  All eigenvalues of $A$ and $B$ are real, and putting them in increasing order minimizes the sum $\sum_i |a_i - b_i|^2$ .

We have $$ \| A- B \| = \sqrt{ \sum_i \sum_j  |a_{ij} - b_{ij}|^2}.$$
In a normalized graph Laplacian, $$a_{ij} = \frac{1}{\sqrt{ \deg(v_i) \deg(v_j)} }$$
if $v_i $ is adjacent to $v_j$, and $a_{ij} = 0$ otherwise.

The link of a $(k-2)$-face (in the $k$-skeleton) has the same distribution as a random graph on the vertices in the link, so standard results give that the degree of every vertex in $\lk_{X}(\tau)$ is exponentially concentrated around its mean $(n-k+1) p^k \ge n^{k \epsilon}$ (see Chapter 3 in \cite{Bollo}) and there are only polynomially many such vertices summed over all links. So w.h.p.\ every vertex in every link has degree $(1+o(1)) np^k \ge n^{k \epsilon}$.  Then the Wielandt--Hoffman theorem tells us that the Frobenius matrix norm of the normalized Laplacian can not shift by more than $O \left( n^{-k \epsilon} \right) = o(1)$ when an edge is deleted.  Hence no single eigenvalue can shift by more than this.

Since we already have $\lambda_2 [ \lk_{X} ( \tau)  ]> 1 -o(1) $ for every $\tau$ by Section \ref{sec:gap}, this gives that $\lambda_2  [ \lk_{X-\sigma} ( \tau)  ] > 1 -o(1) $ for every $\tau$ and $\sigma$ as well.  Applying the cohomology vanishing theorem again, we have that $H^{k-1} ( X - \sigma, \Q) = 0$ for every $k$-dimensional face $\sigma$.

\end{proof}

\begin{proof}[Proof of Part (2) of Theorems \ref{thm:main} and \ref{thm:propT}]
If 
$$\frac{1}{n^{1/k - \epsilon}} \le p \le \left( \frac{\left(\frac{k}{2} + 1-\epsilon \right)\log{n}}{n}\right)^{1/(k+1)},$$
where $0 < \epsilon < 1/ k(k+1)$, then Lemma \ref{lem:del}, together with the excision argument above, gives that w.h.p.\ $H^{k}(X, \Q) \neq 0$.

On the other hand, if 
$$  \left( \frac{k+1 + \epsilon}{n} \right)^{1/k} \le p \le \frac{1}{n^{1 / k - \epsilon} }$$
then an easier argument is available. Indeed, standard results for clique counts give that in this case w.h.p.\ $f_k > f_{k-1} + f_{k+1}$, and the Morse inequalities give that $\beta_k \ge f_k - f_{k-1} - f_{k+1}$, so we conclude that w.h.p.\ $\beta_k > 0$.   Together these two intervals cover the whole range of $p$ for Part (2) of Theorem \ref{thm:main}.

\medskip

Part (2) of Theorem \ref{thm:main} implies Part (2) of \ref{thm:propT} since (T) groups have finite abelianizations.  
\end{proof}

%
%
%
%
%
%
%
%
%

\section{Comments} \label{sec:comments}
%
%
Earlier results of Linial and Meshulam \cite{LM} on Bernoulli random $2$-complexes $Y_2(n,p)$, and more generally of Meshulam and Wallach \cite{MW} on random $d$-complexes $Y_d(n,p)$ also give sharp threshold cohomology-vanishing analogues of the Erd\H{o}s--R\'enyi theorem.  The techniques in these papers are combinatorially intricate cocycle counting arguments. One common thread in this area is the notion of expansion; see for example the discussion of random complexes as higher dimensional expanders in \cite{Dotterrer12}.

DeMarco, Hamm, and Kahn independently proved the $k=1$ case of Theorem \ref{thm:main} with $\Z / 2$-coefficients \cite{DHK12}. This is a slightly stronger result topologically speaking, since $H_1(X, \Z/2)=0$ implies $H_1(X, \Q)=0$ by the universal coefficient theorem, and it would be interesting to know if their techniques can be extended to $k \ge 2$, or to other finite fields.

\medskip
 

One might expect that part (1) of Theorem \ref{thm:main} could be slightly sharpened as follows \footnote{As this article is nearing its final revisions, it seems that an improved version of the spectral gap theorem in Section 3, in joint work with Hoffman and Paquette, will be strong enough to establish Conjecture \ref{con:Poi}.}.
%
%

\begin{conjecture} \label{con:Poi}
If
$$ p = \left( \frac{ \left( \frac{k}{2}+1 \right) \log{n} + \left( \frac{k}{2}  \right) \log\log{n} + c}{n} \right)^{1 / (k+1)},$$ where $k \ge 1$ and $c \in \R$ are constant, then the $k$th Betti number $\beta^k$ converges in law to a Poisson distribution $$\beta^{k} \to \Pois ( \mu)$$ with mean
$$\mu =  \frac{ \left (\frac{k}{2}+1 \right)^{k/2}}{(k+1)!} e^{-c}.$$  
In particular, $$\prob [ H^k (X, \Q) = 0 ] \to \exp \left( - \frac{ (\frac{k}{2}+1)^{k/2}}{(k+1)!} e^{-c}\right)$$
as $n \to \infty$.
\end{conjecture}

In particular, if Conjecture \ref{con:Poi} is true, then letting $c \to \pm \infty$ arbitrarily slowly would give the correct width of the critical window.

Conjecture \ref{con:Poi} should be compared with Theorem \ref{thm:midp}.  The conjecture is equivalent to saying that for the given range of $p$, w.h.p.\ characteristic functions on isolated $k$-faces generate rational cohomology. The analogous statement is well known for random graphs $G(n,p)$, see for example chapter 10 of Bollob\'as \cite{Bollo}.

\medskip

It would also be interesting to know if Corollary \ref{cor:ddim} can be refined to homotopy equivalence, at least for a slightly smaller range of $p$.

\begin{conjecture} \label{con:torsion}  
Let $k \ge 3$ and $\epsilon > 0$ be fixed.  If
$$ \frac{n^{\epsilon} }{n^{1/k}}
 \le p \le \frac{n^{-\epsilon}}{n^{1/(k+1)}},$$
then w.h.p. $X$ is homotopy equivalent to a bouquet of $k$-spheres.
\end{conjecture}

Simplicial complexes and posets in topological combinatorics are often homotopy equivalent to bouquets of $d$-spheres \cite{Forman, Bjorner}, and proving this conjecture might provide a kind of measure-theoretic explanation of the seeming ubiquity of this phenomenon.

Conjecture \ref{con:torsion} is equivalent to showing that integral homology $H_{*}(X, \Z)$ is torsion-free, since (simply connected) Moore spaces are unique up to homotopy equivalence, e.g.\ see example 4.34 in Hatcher \cite{Hatcher}. In contrast, Kalai showed that uniform random $\Q$-acyclic complexes have, on average, enormous torsion groups \cite{Gil}.

\bigskip

For the lower threshold, I would guess that both of the following hold.

\begin{conjecture} \label{con:slight}
Suppose that $k \ge 1$ is fixed.
If $$p =o \left(  \frac{1}{n^{1/k}} \right),$$
then w.h.p.\ $H^k (X, \Q) = 0$.
\end{conjecture}

\begin{conjecture} \label{con:const} If $k \ge 0$ is fixed and 
$$C_k =\inf \left\{ \lambda > 0 \mid p = \frac{\lambda}{n^{1/k} } \implies  \mbox{ w.h.p. } H^k (X, \Q) \neq 0\right\},$$
then $C_k > 0$.
\end{conjecture}

The lower bound in part (2) of Theorem \ref{thm:main} shows that $C_k \le (k+1)^{1/k}$. This can almost certainly be improved; for example $C_1 \le 1$, since cycles appear w.h.p.\ in the random graph $G(n,p)$ once $p \ge 1 /n$.
On the other hand, if $p = c / n$ with $0 < c < 1$ fixed and $G \sim G(n,p)$, then 
$$ \prob[H_1(G)=0] \to \sqrt{1-c} \, \exp(c/2 + c^2/4 )$$
which is strictly positive for $0 < c< 1$, so in fact $C_1 = 1$; see Pittel \cite{P88} for a proof.

In a series of papers, Kozlov \cite{Kozlov}, Cohen et.\ al.\ \cite{CCFK12}, and most recently Aronshtam and Linial \cite{AL13} have studied the threshold for the appearance of $d$-cycles in the Bernoulli random $d$-complex $Y_d (n,p)$. Conjectures \ref{con:slight} and \ref{con:const} are inspired by the results in these papers.



\section*{Acknowledgements}

I thank Noga Alon, Eric Babson, Chris Hoffman, Roy Meshulam, Elliot Paquette, and Uli Wagner for many helpful conversations.  I am especially grateful to Dave Rosoff and to an anonymous referee for careful readings of earlier drafts, and for several suggestions which substantially improved this article.

I first learned of applications of Garland's method in topological combinatorics from \cite{ABM}, where Aharoni, Berger, and Meshulam establish a global analogue of the cohomology vanishing theorem for flag complexes.

This work was completed when I was a member at IAS, 2010--2011, and I am grateful for having had the opportunity to work in such a rich mathematical environment. I especially thank Bob MacPherson for his encouragement and support during this time.

\bibliographystyle{plain}
\bibliography{flagrefs}

\end{document}